\documentclass[a4paper,11pt]{amsart}

\usepackage{amsmath,amsfonts,amsthm,mathrsfs}

\newtheorem{theorem}{Theorem}
\newtheorem{lemma}{Lemma}
\newtheorem{corollary}{Corollary}
\theoremstyle{remark}
\newtheorem{remark}{Remark}
\newtheorem{assumption}{Assumption}

\DeclareMathOperator{\dimH}{\mathrm{dim}_\mathrm{H}}
\DeclareMathOperator{\var}{var}

\title{On Shrinking Targets for Piecewise Expanding Interval Maps}

\author{Tomas Persson}
\address{Centre for Mathematical Sciences, Lund University, Box 118,
  22100 Lund, Sweden}
\email{tomasp@maths.lth.se}
\urladdr{http://www.maths.lth.se/~tomasp}

\author{Micha\l{} Rams}
\address{Instytut Matematyczny, Polska Akademia Nauk, ul. Sniadeckich
  8, 00-656 Warszawa, Poland}
\email{rams@impan.pl}
\urladdr{http://www.impan.pl/~rams}

\begin{document}

\date{\today}

\begin{abstract}
  For a map $T \colon [0,1] \to [0,1]$ with an invariant measure
  $\mu$, we study, for a $\mu$-typical $x$, the set of points $y$ such
  that the inequality $|T^n x - y| < r_n$ is satisfied for infinitely
  many $n$. We give a formula for the Hausdorff dimension of this set,
  under the assumption that $T$ is piecewise expanding and $\mu_\phi$
  is a Gibbs measure. In some cases we also show that the set has a
  large intersection property.
\end{abstract}

\maketitle

\section{Introduction}

We consider a map $T \colon [0,1] \to [0,1]$. Let $r =
(r_n)_{n=1}^\infty$ be a sequence of decreasing positive numbers. In
this paper we shall investigate the size of the set
\begin{align*}
  E (x, r) &= \{\, y \in [0,1] : d (T^n x, y) < r_n
  \text{ for infinitely many } n \,\} \\ &= \limsup_{n \to \infty}
  B(T^n x, r_n).
\end{align*}

Sets of this form with $T \colon x \mapsto 2x \mod 1$ were studied by
Fan, Schmeling and Troubetzkoy in \cite{Fanetal}. Li, Wang, Wu and Xu
studied in \cite{Lietal} a related but different set in the case when
$T$ is the Gau\ss{} map.

In the paper \cite{LiaoSeuret}, Liao and Seuret studied the case when
$T$ is an expanding Markov map with a Gibbs measure $\mu$. They proved
that if $r_n = n^{-\alpha}$, then for $\mu$-almost all $x$, the set $E
(x, r)$ has Hausdorff dimension $1/\alpha$ provided that $1/\alpha$
is not larger than the dimension of the measure $\mu$.

In this paper we will consider more general maps than those studied by
Liao and Seuret and prove results similar to those of the three papers
mentioned above. We will use a method of statistical nature very
similar to the one used in \cite{Persson}. The maps we will work with
are mostly piecewise expanding interval maps, but some of our results
are valid for more abstract maps with certain statistical properties.

We will not assume that the maps have a Markov partition. In the case
that $\mu$ is a measure that is absolutely continuous with respect to
Lebesgue measure, we can consider the sets $E(x,r)$ with $r_n =
n^{-\alpha}$ for any $\alpha > 1$. However, for other measures $\mu$
we have to impose extra restrictions on $\alpha$ and our results are
only valid for sufficiently large $\alpha$. This extra restriction is
not present in the works of Fan, Schmeling and Troubetzkoy; Li, Wang,
Wu and Xu; and Liao and Seuret.

The results of this paper are presented in two main theorems, found in
Sections~\ref{sec:acim} and \ref{sec:gibbs}. The first theorem treats
the case when $\mu$ is absolutely continuous with respect to the
Lebesgue measure and no extra restriction is imposed on $\alpha$. In
this case our result is a generalisations of the corresponding result
by Liao and Seuret and we also prove that for almost all $x$ the set
$E(x, r)$ has large intersections. This means that the set $E(x,r)$
belongs for some $0 < s <1$ to the class $\mathscr{G}^s$ of
$G_\delta$-sets, with the property that any countable intersection of
bi-Lipschitz images of sets in $\mathscr{G}^s$ has Hausdorff dimension
at least $s$. See Falconer's paper \cite{Falconer} for more details
about those classes of sets. The large intersection property was not
proved in any of the papers \cite{Fanetal}, \cite{Lietal} and
\cite{LiaoSeuret}.

The second theorem treats more general measures and is only valid for
sufficiently large $\alpha$. Restriction of this type are not present
in the papers \cite{Fanetal}, \cite{Lietal} and \cite{LiaoSeuret}. We
have not been able to prove the large intersection property in this
case in the general setting. However, we prove that if the map is a
Markov map, then the large intersection property holds.

In Section~\ref{sec:examples} we provide explicit examples of maps
that satisfy the assumptions of the two main theorems. Most examples
are for uniformly expanding maps, but we also give some examples with
non-uniformly expanding maps.

One can also study the Hausdorff dimension of the complement of $E
(x,r)$. That was done both in the paper by Liao and Seuret as well as
that by Fan, Schmeling and Troubetzkoy, but we shall not do so in this
paper.

\section{Maps with Absolutely Continuous Invariant Measures} \label{sec:acim}

We will first work with maps $T \colon [0,1] \to [0,1]$ satisfying the
following assumptions.
\begin{assumption} \label{as:acim}
  There exists an invariant measure $\mu$ that is
  absolutely continuous with respect to Lebesgue measure, with density
  $h$ such that $c_h^{-1} < h < c_h$ holds Lebesgue almost everywhere
  for some constant $c_h > 0$.
\end{assumption}
\begin{assumption} \label{as:decay1}
  Correlations decay with summable speed for functions of bounded
  variation: There is a function $p \colon \mathbb{N} \to (0,\infty)$
  such that if $f \in L^1$ and $g$ is of bounded variation, then
  \[
  \biggl| \int f \circ T^n \cdot g \, \mathrm{d} \mu - \int f \,
  \mathrm{d} \mu \int g \, \mathrm{d} \mu \biggr| \leq \lVert f
  \rVert_1 \lVert g \rVert p(n),
  \]
  where $\lVert \psi \rVert = \lVert \psi \rVert_1 + \var \psi$, and
  we assume that the correlations are summable in the sense that
  \[
  C := \sum_{n=0}^\infty p(n) < \infty.
  \]
\end{assumption}

We prove the following theorem. The proof is in
Section~\ref{sec:absctsproof}.

\begin{theorem} \label{the:shrinkingtarget}
  Under the Assumptions~\ref{as:acim} and \ref{as:decay1} above,
  $\dimH E (x, r) \geq s$ for Lebesgue almost all $x \in [0,1]$, where
  \[
  s = \sup \{\, t : \exists c, \forall n : n^{-2} \sum_{j=1}^n
  r_j^{-t} < c \,\}.
  \]
  Moreover, the set $E(x,r)$ belongs to the class $\mathscr{G}^s$ of
  $G_\delta$-sets with large intersections for Lebesgue almost all
  $x$.

  In particular, if $r_n = n^{-\alpha}$ then $\dimH E (x, r) =
  1/\alpha$ for Lebesgue almost all $x \in [0,1]$.
\end{theorem}

In Section~\ref{sec:examples} we provide some examples of maps
satisfying the assumptions of Theorem~\ref{the:shrinkingtarget}.

\section{Maps with Gibbs Measures} \label{sec:gibbs}

We will now consider a map $T \colon [0,1] \to [0,1]$ with a Gibbs
measure $\mu_\phi$. Our assumptions are as follows.

\begin{assumption} \label{as:piecewise}
  $T$ is piecewise monotone and expanding with respect to a finite
  partition, and there is bounded distortion for the derivative $T'$.
\end{assumption}

\begin{assumption} \label{as:gibbs}
  The potential $\phi \colon [0,1] \to \mathbb{R}$ is of bounded
  distortion, and there is a Gibbs measure $\mu_\phi$ to the potential
  $\phi$, with $\mu_\phi = h_\phi \nu_\phi$ where $h_\phi$ is a
  bounded function that is bounded away from zero, and $\nu_\phi$ is a
  conformal measure, that is, for any subset $A$ of a partition
  element holds
  \[
    \nu_\phi (T(A)) = \int_A e^{P(\phi)-\phi} \, \mathrm{d} \nu_\phi,
  \]
  where $P (\phi)$ denotes the topological pressure of $\phi$.
\end{assumption}

\begin{assumption} \label{as:decay2}
  We have summable decay of correlations for functions of bounded
  variation. That is we assume that there is a function $p \colon
  \mathbb{N} \to (0,\infty)$ such that if $f \in L^1 (\mu_\phi)$ and
  $g$ is of bounded variation, then
  \[
  \biggl| \int f \circ T^n \cdot g \, \mathrm{d} \mu_\phi - \int f
  \mathrm{d} \mu_\phi \int g \, \mathrm{d} \mu_\phi \biggr| \leq
  \lVert f \rVert_1 \lVert g \rVert p(n),
  \]
  holds for all $n$, and we assume that
  \[
  C := \sum_{n=0}^\infty p(n) < \infty.
  \]
\end{assumption}

\begin{assumption} \label{as:measureofball}
  There is a number $s_0 > 0$ such that for any $s < s_0$ there is a
  constant $c_s$ such that $\mu_\phi (I) \leq c_s |I|^s$ holds for any
  interval $I \subset [0,1]$.
\end{assumption}

\begin{remark} \label{rem:dimension}
  We note that Assumption~\ref{as:measureofball} implies that
  \begin{equation} \label{eq:localenergy}
    \iint | x - y |^{-t} \, \mathrm{d} \mu_\phi (x) \mathrm{d}
    \mu_\phi (y) \leq \frac{t c_s}{s-t}.
  \end{equation}
  for any $t < s < s_0$.  This follows since, for any $x$, we have
  \begin{multline} \label{eq:energyatx}
    \int |x - y|^{-t} \, \mathrm{d} \mu_\phi (y) = \int_1^\infty
    \mu_\phi (B (x, u^{-1/t})) \, \mathrm{d}u \\ \leq \int_1^\infty
    c_s u^{-s/t} \, \mathrm{d} u = \frac{t c_s}{s-t},
  \end{multline}
  which implies \eqref{eq:localenergy}.

  Note also that \eqref{eq:localenergy} implies that the lower pointwise
  dimension of $\mu_\phi$ is at least $s_0/2$ at any point in
  $[0,1]$. Indeed, since $|I|^{-s} \leq |x - y|^{-s}$ holds whenever
  $x, y \in I$, we have together with \eqref{eq:localenergy} that
  \begin{align*}
    |I|^{-s} \mu_\phi (I)^2 &\leq \iint_{I \times I} | x - y |^{-s} \,
    \mathrm{d} \mu_\phi (x) \mathrm{d} \mu_\phi (y) \\ & \leq \iint |
    x - y |^{-s} \, \mathrm{d} \mu_\phi (x) \mathrm{d} \mu_\phi (y) =
    c
  \end{align*}
  holds whenever $s < s_0$. Hence $\mu_\phi (I) \leq \sqrt{c}
  |I|^{s/2}$ and the claim follows.
\end{remark}

In this setting we can prove a similar result to
Theorem~\ref{the:shrinkingtarget}. The proof of the following theorem
is in Section~\ref{sec:gibbsproof}.

\begin{theorem} \label{the:gibbs}
  Assume that $T \colon [0,1] \to [0,1]$ satisfies the
  Assumptions~\ref{as:piecewise}, \ref{as:gibbs}, \ref{as:decay2} and
  \ref{as:measureofball}. Then, we have that $\dimH E(x,r) \geq s$ for
  $\mu_\phi$-almost all $x$, where
  \[
  s = \sup \{\, t < s_0 : \exists c, \forall n : n^{-2} \sum_{j=1}^n
  r_j^{-t} < c \,\}.
  \]
  In particular, if $r_n = n^{-\alpha}$ and $\alpha > 1 / s_0$, then
  $\dimH E (x,r) = 1/\alpha$ for $\mu_\phi$-almost every $x$.
\end{theorem}

\begin{remark}
  Note that if $\alpha \leq 1/s_0$ then Theorem~\ref{the:gibbs} gives us
  the result that $\dimH E (x,r) \geq s_0$. However, one would expect
  that $\dimH E(x,r) = 1/\alpha$ as long as $1/\alpha$ is not larger
  than the dimension of $\mu_\phi$, which is the result proved by Liao
  and Seuret in their setting.

  As is clear from Remark~\ref{rem:dimension}, our method cannot work
  for the full range of $\alpha$, since we rely on
  Assumption~\ref{as:measureofball}, so that we cannot consider
  $\alpha$ such that $(2\alpha)^{-1}$ is larger than the lower
  pointwise dimension of $\mu_\phi$ at any point.
\end{remark}

If we also assume that the map is Markov, then we can prove the large
intersection property of the set $E (x,r)$.

\begin{theorem} \label{the:markov}
  Assume that $T \colon [0,1] \to [0,1]$ is a Markov map that
  satisfies the Assumptions~\ref{as:piecewise}, \ref{as:gibbs},
  \ref{as:decay2} and \ref{as:measureofball}. Then, we have that
  $E(x,r) \in \mathscr{G}^s$ for $\mu_\phi$-almost all $x$, where
  \[
  s = \sup \{\, t < s_0 : \exists c, \forall n : n^{-2} \sum_{j=1}^n
  r_j^{-t} < c \,\}.
  \]
\end{theorem}

In the next section we give examples of maps satisfying the
assumptions of Theorem~\ref{the:gibbs}.

\section{Examples} \label{sec:examples}

\subsection{Examples for Theorem~\ref{the:shrinkingtarget}} \label{ssec:example1}
There exist some dynamical systems that obviously satisfy the
assumptions of Theorem \ref{the:shrinkingtarget}, for example $n-1$
expanding diffeomorphisms of the circle. We are going to present less
obvious examples of application of our results.

For instance, the maps studied by Liverani in \cite{Liverani} satisfy
the assumptions of Theorem~\ref{the:shrinkingtarget}. These maps are
defined as follows. Assume that there is a finite partition
$\mathscr{P}$ of $[0,1]$ into intervals, such that on every interval
$I \in \mathscr{P}$, the map $T$ can be extended to a $C^2$ map on a
neighbourhood of the closure of $I$, and assume that there is a
$\lambda > 1$ such that $|T'| \geq \lambda$ holds everywhere. To put
it shortly, $T$ is piecewise $C^2$ with respect to a finite partition,
and uniformly expanding. We assume also that $T$ is weakly covering,
as defined by Liverani: The map $T$ is said to be weakly covering if
there exists an $N_0 \in \mathbb{N}$ such that if $I \in \mathscr{P}$,
then
\[
 \bigcup_{k=0}^{N_0} T^k (I) \supset [0,1] \setminus W,
\]
where $W$ is the set of points that never hit the discontinuities of
$T$. Under the assumptions mentioned above, it is shown in
\cite{Liverani} that $T$ has an invariant measure $\mu$ satisfying the
Assumption~\ref{as:acim} above, and the correlations decay
exponentially. Hence they are summable and Assumption~\ref{as:decay1}
holds. We therefore have the following corollary.

\begin{corollary}
  If $T \colon [0,1] \to [0,1]$ is piecewise $C^2$ with respect to a
  finite partition, uniformly expanding, and weakly covering, then
  with $r_n = n^{-\alpha}$, $\alpha \geq 1$ we have
  \[
  \dimH E (x,r) = \frac{1}{\alpha}
  \]
  and $E (x,r) \in \mathscr{G}^{1/\alpha}$ for Lebesgue almost every
  $x \in [0,1]$.
\end{corollary}

In fact, it is not necessary to assume that the map is piecewise
$C^2$. It is sufficient that the derivative is of bounded variation,
since then one can combine the estimates by Rychlik \cite{Rychlik}
with the method of Liverani \cite{Liverani} to get the same result.

If the map is piecewise expanding with an indifferent fixed point,
then Assumption~\ref{as:decay1} does not hold. However, as we will see
below, we can still use Theorem~\ref{the:shrinkingtarget} to get the
following result.

\begin{corollary}
  Let $T_\beta:[0,1)\to [0,1)$ with $\beta > 1$ be the
  Manneville--Pomeau map
  \[
  x \mapsto \left\{ \begin{array}{ll} x + 2^{\beta-1} x^{\beta} & x <
    1/2 \\ 2x - 1 & x \geq 1/2 \end{array} \right.
  \]
  and $r_j = j^{-\alpha}, \alpha \geq 1$.  If\/ $1 < \beta < 2$, then
  for Lebesgue almost every $x$ we have that $\dim_H E(x,r) =
  1/\alpha$ and $E (x,r) \in \mathscr{G}^{1/\alpha}$.  If $\beta \geq
  2$, then for Lebesgue almost every $x$ we have that $\dimH E(x,r) =
  \frac{1}{\alpha (\beta - 1)}$ and $E (x,r) \in \mathscr{G}^{1/\alpha
    /(\beta - 1)}$.
\end{corollary}

\begin{proof}
  Let $S_\beta$ be the first return map on the interval $[1/2, 1)$.
    Then there exists an $S_\beta$-invariant measure $\nu$ that is
    absolutely continuous with respect to Lebesgue measure, and $\nu$
    is ergodic.

  Let $R(x)$ be the return time of $x$ to $[1/2, 1)$, that is, we have
    $T_\beta^{R(x)} = S_\beta (x)$.

  It the case $1 < \beta < 2$ we will do as follows. In this case $R$
  is integrable and so, for almost all $x$ there is a constant $c > 0$
  such that
  \begin{equation} \label{eq:returntimeestimate}
    n \leq \sum_{k=1}^n R_k (x) \leq c n
  \end{equation}
  for all sufficiently large $n$. (The lower bound always holds, since
  $R \geq 1$.) We put
  \[
  r_j' = (cj)^{-\alpha} \qquad \text{and} \qquad r_j'' = j^{-\alpha}.
  \]
  Then for almost all $x$ we will have that
  \[
    B (S_\beta^j (x), r_j') \subset B (T_\beta^{\sum_{k=1}^j R_k (x)}
    (x), r_{\sum_{k=1}^j R_k (x)}) \subset B(S_\beta^j (x), r_j'')
  \]
  for sufficiently large $j$. Hence, with
  \begin{align*}
  E' (x,r') &:= \limsup_{j \to \infty} B (S_\beta^j (x), r_j'), \\
  E'' (x,r'') &:= \limsup_{j \to \infty} B (S_\beta^j (x), r_j''),
  \end{align*}
  we have
  \[
  E' (x,r') \subset E (x,r) \cap [1/2,1] \subset E''(x,r'')
  \]
  for almost all $x$.

  Now, Theorem~\ref{the:shrinkingtarget} implies that $E' (x,r') \cap
  [1/2, 1) \in \mathscr{G}^{1/\alpha}$ for almost all $x$ and $\dimH
    E'' (x,r'') = 1/\alpha$ for almost all $x$. This implies the
    desired result for $E(x,r) \cap [1/2, 1)$.

  In the same way we can get the result for $E (x,r) \cap I_n$ where
  $I_n = [x_n, 1)$, where $x_n$ is the $n$-th pre-image of $1/2$ with
    respect to the left branch of $T_\beta$. This concludes the proof
    for the case $1 \leq \beta < 2$.

  The method above does not quite work when $\beta \geq 2$, since then
  $\int R\, \mathrm{d} \nu = \infty$, and the upper bound of
  \eqref{eq:returntimeestimate} fails. However, whenever $\varepsilon
  > 0$, we have for almost all $x$ that
  \[
  n^{\beta - 1 - \varepsilon} \leq \sum_{k = 1}^n R_k (x) \leq
  n^{\beta - 1 + \varepsilon}.
  \]
  holds for large $n$. The upper bound above follows from
  Theorem~2.3.1 of \cite{Aaronson}. The lower bound follows using
  Theorem~1 in \cite{AaronsonDenker}.

  We now proceed as in the case $1 < \beta < 2$. Put
  \[
  r_j' = (c_2 j)^{-\alpha(\beta - 1 + \varepsilon)} \qquad \text{and}
  \qquad r_j'' = (c_1 j)^{-\alpha (\beta - 1 - \varepsilon)}.
  \]
  With the same notation as previously we then have that
  \[
  E' (x,r') \subset E (x,r) \cap [1/2,1] \subset E''(x,r'')
  \]
  for almost all $x$.

  Theorem~\ref{the:shrinkingtarget} implies that for almost all $x$
  $E'(x, r') \cap [1/2, 1) \in \mathscr{G}^{1/\alpha/(\beta - 1 +
      \varepsilon)}$ and $\dimH E''(x, r'') \cap [1/2,1) = (\alpha
      (\beta - 1 - \varepsilon))^{-1}$. Since $\varepsilon > 0$ can be
      chosen arbitrarily small, this implies the result for $E (x,r)
      \cap [1/2,1)$. As before, we get the result stated in the
        corollary by considering $E (x,r \cap I_n)$ in the same way.
\end{proof}

\subsection{Examples for Theorem~\ref{the:gibbs}}

Here we will show that the Assumptions~\ref{as:decay2} and
\ref{as:measureofball} are satisfied for a natural class of
systems. Consider a map $T$ which is piecewise $C^2$ with respect to a
finite partition, and uniformly expanding, as defined in
Section~\ref{ssec:example1}. Then Assumption~\ref{as:piecewise} is
satisfied.

Suppose that $\phi$ satisfies the assumptions of Liverani, Saussol and
Vaienti in \cite{Liveranietal}, that is, $e^\phi$ is of bounded
variation and that there exists an $n_0$ such that
\begin{equation} \label{eq:contractingpotential}
  \sup e^{S_{n_0} \phi} < \inf L_\phi^{n_0} 1,
\end{equation}
where $S_{n_0} \phi = \phi + \phi \circ T + \cdots + \phi \circ
T^{n_0-1}$ and
\[
L_\phi f (x) = \sum_{T(y) = x} e^{\phi (y)} f(y)
\]
is the transfer operator with respect to the potential $\phi$. We
assume moreover that $\phi$ is piecewise $C^2$ with respect to the
partition of the map, so that the bounded distortion part of
Assumption~\ref{as:gibbs} is satisfied.

Finally, we assume that $T$ is covering, in the sense that for any non
trivial interval $I$ there is an $n$ such that $T^n (I) \supset [0,1]
\setminus W$, where $W$ is the set of points that never hit the
discontinuities of $T$.  Under these assumptions, there exists a
unique Gibbs measure $\mu_\phi$ and the Assumptions~\ref{as:gibbs} and
\ref{as:decay2} hold, see Theorem 3.1 in \cite{Liveranietal}. In this
setting, Assumption~\ref{as:measureofball} will also be satisfied.

\begin{corollary} 
  Assume that $T \colon [0,1] \to [0,1]$ is piecewise $C^2$ with
  respect to a finite partition, uniformly expanding and covering. If
  $\phi$ satisfies the assumptions above, then the
  Assumption~\ref{as:measureofball} is satisfied with
  \[
  s_0 =
  \limsup_{m \to \infty} \inf \frac{S_m \phi - m P(\phi)}{-\log
    |(T^m)'|}.
  \]
  Hence, if $r_n = n^{-\alpha}$, $\alpha > 1/s_0$, then
  \[
  \dimH E(x,r) = \frac{1}{\alpha}
  \]
  for $\mu_\phi$-almost every $x \in [0,1]$.
\end{corollary}

\begin{proof}
  We will rely on the part of Assumption~\ref{as:gibbs} that says that
  if $A$ is a subset of one of the partition elements, then
  \begin{equation} \label{eq:conformal}
    \nu_\phi (T(A)) = \int_A e^{P(\phi)-\phi} \, \mathrm{d} \nu_\phi,
  \end{equation}
  where $P (\phi) = \lim_{n \to \infty} n^{-1} \log \inf L_\phi^n 1$
  denotes the topological pressure of $\phi$. Since there are
  constants $c_1$ and $c_2$ such that $0 < c_1 < h < c_2$, it suffices
  to prove Assumption~\ref{as:measureofball} for the measure
  $\nu_\phi$.

  Let $r_0 > 0$ be such that any interval of length $r_0$ intersects
  at most two partition elements. If $r < r_0$ and $I$ is an interval
  of length $r$, then $I$ intersects at most two different partition
  elements and therefore $T(I)$ consists of at most two intervals of
  length at most $r \sup |T'|$. By \eqref{eq:conformal}, it follows
  that
  \[
  \nu_\phi (I) \inf_I e^{P(\phi) - \phi} \leq 2 \sup_{|I_1| = r \sup_I
    |T'|} \nu (I_1).
  \]
  Hence
  \[
  \nu_\phi (I) \leq 2 \sup_I e^{\phi - P(\phi)} \sup_{|I_1| = r \sup_I
    |T'|} \nu (I_1).
  \]
  By induction, we conclude that
  \[
  \nu_\phi (I) \leq \bigl(2 \sup_I e^{\phi - P(\phi)} \bigr)^n,
  \]
  where $n$ is the largest integer such that $r (\sup_I |T'|)^n \leq
  r_0$. Hence we have that there is a constant $C_1$, that does not
  depend on $I$, such that
  \[
  \nu_\phi (I) \leq C_1 r^{\theta_1} = C_1 |I|^{\theta_1}, \quad
  \theta_1 = \frac{\log 2 + \log \sup_I e^{\phi - P(\phi)}}{- \log
    \sup_I |T'|}
  \]
  By making the constant $C_1$ sufficiently large, we can ensure that
  the estimate above holds for all intervals $I$, not only those that
  are sufficiently small.

  By considering $T^m$ instead of $T$, where $m$ is a positive
  integer, the same argument gives us the existence of a constant $C_m$
  such that
  \[
  \nu_\phi (I) \leq C_m |I|^{\theta_m}, \quad \theta_m = \frac{\log 2
    + \log \sup_I e^{S_m \phi - m P(\phi)}}{- \log \sup_I |(T^m)'|}
  \]
  holds for any interval $I$.

  This shows that we may take $s_0 = \limsup_{m \to \infty} \inf
  \frac{S_m \phi - m P(\phi)}{-\log |(T^m)'|}$. The assumption
  \eqref{eq:contractingpotential} guaranties that $s_0 > 0$.
\end{proof}

\section{Proof of Theorem~\ref{the:shrinkingtarget}} \label{sec:absctsproof}

The proof of Theorem~\ref{the:shrinkingtarget} will be based on the
following lemma. It is a special case of Theorem~1 in
\cite{PerssonReeve}. We refer to \cite{PerssonReeve} for a proof.

\begin{lemma} \label{lem:frostman}
  Let $E_n$ be open subsets of $[0,1]$, and $\mu_n$ Borel probability
  measures with support in $E_n$, that converge weakly to a measure
  $\mu$ that is absolutely continuous with respect to Lebesgue measure
  and with density that is bounded and bounded away for zero. Suppose
  there exists a constant $C$ such that
  \[
  \iint |x-y|^{-s} \, \mathrm{d} \mu_n (x) \mathrm{d} \mu_n (y) < C
  \]
  holds for all $n$. Then the set $\displaystyle \limsup_{n \to
    \infty} E_n$ belongs to the class $\mathscr{G}^s$ and has
  Hausdorff dimension at least $s$.
\end{lemma}

We will also make use of the following two lemmata.

\begin{lemma} \label{lem:energy}
  Let $0 < s < 1$. There is a constant $c_s > 0$ such that if $B_1 =
  B(x_1, r_1)$ and $B_2 = B(x_2,r_2)$ are two balls, then
  \[
  \frac{1}{r_1 r_2} \int_{B_1} \int_{B_2} |x - y|^{-s} \, \mathrm{d}x
  \mathrm{d}y \leq c_s \min \{|x_1 - x_2|^{-s}, r_1^{-s}, r_2^{-s} \},
  \]
  and for any fixed $x_2$, the variation of the function
  \[
  x_1 \mapsto \frac{1}{r_1 r_2} \int_{B_1} \int_{B_2} |x - y|^{-s} \,
  \mathrm{d}x \mathrm{d}y,
  \]
  is less than $2 c_s \min \{r_1^{-s}, r_2^{-s}\}$.
\end{lemma}

\begin{proof}
  This is intuitively clear, but we provide a proof.

  We suppose that $r_1 \geq r_2$. Let
  \[
  I (x_1,x_2) = \frac{1}{r_1 r_2} \int_{B_1} \int_{B_2} |x - y|^{-s}
  \, \mathrm{d}x \mathrm{d}y.
  \]
  It is clear that $I$ achieves it's maximal value when $x_1 = x_2$,
  for instance when $x_1 = x_2 = 1/2$. Then a direct calculation shows
  that there is a constant $c_1$ such that
  \[
  I (1/2, 1/2) \leq c_1 r_1^{-s}.
  \]
  Hence $I (x_1,x_2) \leq c_1 r_1^{-s} = c_1 \min \{r_1^{-s}, r_2^{-s}
  \}$.

  Suppose that $|x_1 - x_2| > r_1$. It suffices to show that
  $I (x_1,x_2) \leq c_2 |x_1 - x_2|^{-s}$ holds for some constant
  $c_2$. By a change of variables, we have that
  \begin{align*}
    I (x_1, x_2) &= |x_1 - x_2|^{-s} \int_{-1}^1 \int_{-1}^1 \Bigl| 1
    - \frac{r_1}{|x_1-x_2|} u - \frac{r_2}{|x_1 - x_2|} v \Bigr|^{-s}
    \, \mathrm{d}u \mathrm{d}v \\ &\leq 4 |x_1 - x_2|^{-s} \int_0^1
    \int_0^1 \Bigl| 1 - \frac{r_1}{|x_1-x_2|} u - \frac{r_2}{|x_1 -
      x_2|} v \Bigr|^{-s} \, \mathrm{d}u \mathrm{d}v.
  \end{align*}
  Since $r_1 / |x_1 - x_2|$ and $r_2 / |x_1 - x_2|$ are not larger
  than $1$, we have that
  \[
  I (x_1, x_2) \leq 4 |x_1 - x_2|^{-s} \int_0^1 \int_0^1 | 1 - u - v
  |^{-s} \, \mathrm{d}u \mathrm{d}v = c_2 |x_1 - x_2|^{-s}.
  \]
  We can now conclude that $I (x_1, x_2) \leq c_s \min \{|x_1 -
  x_2|^{-s}, r_1^{-s}, r_2^{-s} \}$, with $c_s = \max \{c_1, c_2 \}$.

  The statement about the variation is now a direct consequence since
  the function
  \[
  x_1 \mapsto \frac{1}{r_1 r_2} \int_{B_1} \int_{B_2} |x - y|^{-s} \,
  \mathrm{d}x \mathrm{d}y,
  \]
  is positive, unimodal and with maximal value at most $c_s \min
  \{r_1^{-s}, r_2^{-s}\}$.
\end{proof}

\begin{lemma} \label{lem:decay}
  Suppose that $F \colon [0,1]^2 \to \mathbb{R}$ is a continuous and
  non-negative function, and that $D$ and $E$ are constants such that
  for each fixed $x$ the function $f \colon y \mapsto F(x,y)$
  satisfies $\var f \leq D$ and $\int f \, \mathrm{d} \mu \leq
  E$. Then
  \[
  \int F(T^n x, x) \, \mathrm{d} \mu (x) \leq E + (D + E) p(n).
  \]
\end{lemma}

\begin{proof}
  Let $\varepsilon > 0$. Let $I_k = [k/m, (k+1)/m)$. There is an $m$
  such that if
  \[
  G (x, y) = \sum_{k = 0}^{m-1} F (k/m, y) {1}_{I_k} (x),
  \]
  where ${1}_{I_k}$ denotes the indicator function on $I_k$, then
  \[
  | F (x,y) - G(x,y) | < \varepsilon.
  \]
  Hence we have
  \[
  \biggl| \int F (T^n x, x) \, \mathrm{d} \mu (x) - \int G (T^n x, x)
  \, \mathrm{d} \mu (x) \biggr| < \varepsilon.
  \]

  For each term $F (k/m, y) {1}_{I_k} (x)$ in the sum defining $G$, we
  have
  \begin{multline*}
    \biggl| \int F (k/m, x) {1}_{I_k} (T^n x) \, \mathrm{d} \mu (x) -
    \int F (k/m, x) \, \mathrm{d} \mu (x) \int 1_{I_k} \, \mathrm{d}
    \mu \biggr| \\ \leq \mu (I_k) (D + E) p(n).
  \end{multline*}
  by the decay of correlations. As a consequence, we have
  \[
  \int F (k/m, x) {1}_{I_k} (T^n x) \, \mathrm{d} \mu (x) \leq E \mu
  (I_k) + \mu (I_k) (D + E) p(n).
  \]
  and so
  \[
  \int F (T^n x, x) \, \mathrm{d} \mu (x) \leq \varepsilon + \int G
  (T^n x, x) \, \mathrm{d} \mu (x) \leq \varepsilon + E + (D + E)
  p(n).
  \]
  Let $\varepsilon \to 0$.
\end{proof}

\begin{proof}[Proof of Theorem~\ref{the:shrinkingtarget}]
  Let $B_n (x) = B(T^n x, r_n)$. We consider the sets
  \[
  V_n (x) = \bigcup_{k = m(n)}^n B_k (x)
  \]
  where $m(n)$ is a slowly increasing sequence such that $m(n) < n$
  and $m(n) \to \infty$ as $n \to \infty$. It then holds that $\limsup
  V_n (x) = \limsup B_n (x)$.

  We define probability measures $\mu_{n, x}$ with support in $V_n
  (x)$ by
  \[
  \mu_{n,x} = \frac{1}{n - m(n) + 1} \sum_{k = m(n)}^n \lambda_{B_k
    (x)},
  \]
  where $\lambda_A$ denotes the Lebesgue measure restricted to the set
  $A$ and normalised so that $\lambda_A (A) = 1$. It is clear that
  $\mu_{n,x}$ converges weakly to $\mu$ as $n \to \infty$ for almost
  every $x$.

  We shall consider the quantities
  \[
  I_s (\mu_{n,x}) = \iint |y - z|^{-s} \, \mathrm{d} \mu_{n,x} (y)
  \mathrm{d} \mu_{n,x} (z).
  \]
  From the definition of the measure $\mu_{n,x}$ it follows that
  \[
  I_s (\mu_{n,x}) = \frac{1}{(n - m(n) + 1)^2} \sum_{i = m(n)}^n
  \sum_{j = m(n)}^n \frac{1}{4 r_i r_j} \int_{B_i} \int_{B_j} |y -
  z|^{-s} \, \mathrm{d}y \mathrm{d}z,
  \]
  We now assume that $m(n) < n / 2$. Together with
  Lemma~\ref{lem:energy} we then get that
  \[
  I_s (\mu_{n,x}) \leq \frac{4 c_s}{n^2} \sum_{m(n) \leq i \leq j \leq
    n} \min \{|T^i x - T^j x|^{-s}, r_i^{-s} \}.
  \]

  Using that $\mu$ is $T$-invariant, we can write
  \[
  \int I_s (\mu_{n,x}) \, \mathrm{d}\mu (x) \leq \frac{4 c_s}{n^2}
  \sum_{m(n) \leq i \leq j \leq n} \int \min \{|T^{j -i} x - x|^{-s},
  r_i^{-s} \wedge r_j^{-s} \} \, \mathrm{d} \mu (x),
  \]
  where $a \wedge b$ denotes the minimum of $a$ and $b$.

  An application of Lemma~\ref{lem:decay} gives that
  \begin{align*}
    \int I_s (\mu_{n,x}) \, \mathrm{d}\mu (x) & \leq \frac{1}{n^2}
    \sum_{m(n) \leq i \leq j \leq n} \bigl( C_1 + (C_1 + 2 (r_i^{-s}
    \wedge r_j^{-s}) ) p(j-i) \bigr) \\ & \leq \frac{1}{n^2}
    \sum_{m(n) \leq i \leq j \leq n} C_2 (1 + (r_i^{-s} \wedge
    r_j^{-s}) p(j-i)) \\ & \leq C_2 + \frac{C_2}{n^2} \sum_{j = 1}^n
    \sum_{i=1}^j (r_i^{-s} \wedge r_j^{-s}) p (j-i).
  \end{align*}
  Since $p$ is summable, we can estimate that
  \[
  \sum_{j = 1}^n \sum_{i=1}^j (r_i^{-s} \wedge r_j^{-s}) p (j-i) \leq
  \sum_{j = 1}^n \sum_{i=1}^j r_j^{-s} p (j-i) = \sum_{j = 1}^n
  \sum_{i=0}^{j-1} r_j^{-s} p (i) \leq C \sum_{j=1}^n r_j^{-s}.
  \]
  (This estimate is actually not too rough, since
  \[
  \sum_{j = 1}^n \sum_{i=1}^j (r_i^{-s} \wedge r_j^{-s}) p (j-i) =
  \sum_{j = 1}^n \sum_{i=0}^{j-1} (r_{j-i}^{-s} \wedge r_j^{-s}) p (i)
  \geq \sum_{j = 1}^n r_j^{-s} p(0),
  \]
  which is of the same order of magnitude if $r_i \to 0$ as $i \to \infty$.)

  We conclude that
  \[
  \int I_s (\mu_{n,x}) \, \mathrm{d} \mu (x) \leq C_2 + \frac{C
    C_2}{n^2} \sum_{j=1}^n r_j^{-s},
  \]
  and this is uniformly bounded for all $n$ if
  \[
  s < \sup \{\, t : \exists c, \forall n : n^{-2} \sum_{j=1}^n
  r_j^{-t} < c \,\}.
  \]

  Suppose $s$ satisfies the inequality above. Then, by Birkhoff's
  ergodic theorem, for $\mu$-almost all $x$ the measures $\mu_{n,x}$
  converges weakly to the measure $\mu$, and, as follows from the
  considerations above, for $\mu$-almost all $x$, there is a sequence
  $n_k$, with $n_k \to \infty$, such that the sequence $(I_s
  (\mu_{n_k,x}))_{k=1}^\infty$ is bounded.  We can now apply
  Lemma~\ref{lem:frostman} and conclude that for $\mu$-almost all $x$
  the set $E(x, r)$ belongs to the class $\mathscr{G}^s$. This proves
  the first part of Theorem~\ref{the:shrinkingtarget}.

  If $r_n = n^{-\alpha}$, then it is easy to check that the result
  above gives us that the set $E (x,r)$ belongs to
  $\mathscr{G}^{1/\alpha}$ for almost all $x$. A simple covering
  argument shows that in fact the dimension is not larger than
  $1/\alpha$.
\end{proof}

\section{Proof of Theorems~\ref{the:gibbs} and \ref{the:markov}} \label{sec:gibbsproof}

Assume that we have a sequence of open sets $E_n$, such that each
$E_n$ is a finite union of disjoint intervals, and that the diameters
of these intervals go to zero as $n$ grows. We are first going to
study the Hausdorff dimension of the set $\limsup E_n$ in the
following lemmata. The proof of Theorem~\ref{the:gibbs} will then be
similar to that of Theorem~\ref{the:shrinkingtarget}, but will instead
be based on the lemmata below.

\begin{lemma} \label{lem:localfrostman}
  Let $E_n$ be open subsets of $[0,1]$. Suppose there are Borel
  probability measures $\mu_n$ with support in $E_n$, that converge
  weakly to a measure $\mu$ that satisfies \eqref{eq:localenergy}. If
  for some $t < s < s_0$ there is a constant $C$ such that
  \[
  \iint |x - y|^{-s} \, \mathrm{d} \mu_n (x) \mathrm{d} \mu_n (y) < C
  \]
  for all $n$, then, whenever $I$ is an interval with
  \[
  \iint_{I \times I} |x - y|^{-t} \, \mathrm{d} \mu (x) \mathrm{d}
  \mu (y) < c |I|^{-t} \mu (I)^2,
  \]
  there is an $n_I$ such that
  \[
  \sum |U_k|^t \geq \frac{1}{2c} |I|^t
  \]
  holds for any cover $\{U_k\}$ of $E_n \cap I$, $n > n_I$.
\end{lemma}

\begin{proof}
  The assumptions implies that for any $t < s$
  \[
  \iint_{I \times I} |x - y|^{-t} \, \mathrm{d} \mu_n (x) \mathrm{d}
  \mu_n (y) \to \iint_{I \times I} |x - y|^{-t} \, \mathrm{d} \mu (x)
  \mathrm{d} \mu (y),
  \]
  as $n \to \infty$. (See Corollary 2.3 of \cite{PerssonReeve}.)

  For a measure $\nu$ on $I$ we write $R_t \nu (x) = \int |x-y|^{-t}
  \, \mathrm{d} \nu (y)$.

  Take an interval $I \subset [0,1]$ satisfying the assumption of the
  lemma, and define the measure $\nu_n$ on $I$ by
  \begin{equation} \label{eq:nudefinition}
  \nu_n (A) = \frac{\int_A (R_t \mu_n |_I)^{-1} \, \mathrm{d}
    \mu_n}{\int_I (R_t \mu_n |_I)^{-1} \, \mathrm{d} \mu_n},
  \end{equation}
  where $\mu_n |_I$ denotes the restriction of $\mu_n$ to $I$.

  There is an $n_I$ such that if $n > n_I$ then
  \begin{equation} \label{eq:nuestimate}
  \nu_n (U) \leq 2 c \frac{|U|^t}{|I|^t}
  \end{equation}
  holds for all intervals $U \subset I$. This is proved as follows. By
  the definition of $\nu_n$ the estimate \eqref{eq:nuestimate} is
  equivalent to
  \[
  \frac{1}{|U|^t} \int_U (R_t \mu_n |_I)^{-1} \, \mathrm{d} \mu_n \leq
  \frac{2 c}{|I|^t} \int_I (R_t \mu_n |_I)^{-1} \, \mathrm{d} \mu_n.
  \]
  We prove the stronger statement that
  \begin{equation} \label{eq:Rtintegrals}
  \frac{1}{|U|^t} \int_U (R_t \mu_n |_I)^{-1} \, \mathrm{d} \mu_n \leq
  1 \leq \frac{2 c}{|I|^t} \int_I (R_t \mu_n |_I)^{-1} \, \mathrm{d}
  \mu_n.
  \end{equation}

  The first inequality in \eqref{eq:Rtintegrals} is proved in
  \cite{PerssonReeve}. (Use Lemma~2.4 of \cite{PerssonReeve} and
  approximate with measures that are absolutely continuous with
  respect to Lebesgue.) To prove the second inequality we use Jensen's
  inequality and Assumption~\ref{as:measureofball} (in particular
  \eqref{eq:localenergy}) to conclude that
  \begin{align*}
    \int_I (R_t \mu_n |_I)^{-1} \, \frac{\mathrm{d} \mu_n}{\mu_n (I)}
    & \geq \biggl(\int_I (R_t \mu_n |_I) \, \frac{\mathrm{d}
      \mu_n}{\mu_n (I)} \biggr)^{-1} \\ &= \biggl( \frac{1}{\mu_n (I)}
    \iint_{I \times I} |x-y|^{-t} \mathrm{d} \mu_n (x) \mathrm{d}
    \mu_n (y) \biggr)^{-1} \\ &\geq \frac{1}{\sqrt 2} \biggl(
    \frac{1}{\mu (I)} \iint_{I \times I} |x-y|^{-t} \mathrm{d} \mu (x)
    \mathrm{d} \mu (y) \biggr)^{-1} \\ &\geq \frac{1}{2c}
    |I|^t \mu_n (I)^{-1},
  \end{align*}
  provided $n > n_I$ for some $n_I$. Hence
  \[
  \frac{1}{|I|^t} \int_I (R_t \mu_n |_I)^{-1} \, \mathrm{d} \mu_n\geq
  \frac{1}{2c}
  \]
  and \eqref{eq:Rtintegrals} follows.

  We have now proved \eqref{eq:nuestimate}, and will use it as
  follows. Suppose that $\{U_k\}$ is a cover of $E_n \cap I$, and $n >
  n_I$. Then
  \[
  1 = \nu_n (\bigcup_k U_k) \leq \sum_k \nu_n (E_k) \leq
  \frac{2 c}{|I|^t} \sum_k |U_k|^t.
  \]
  This shows that $\sum_k |U_k|^t \geq \frac{1}{2c} |I|^t$ for
  any cover $\{U_k\}$ of $E_n \cap I$.
\end{proof}

If we would have known that for some constant $c$, the estimate
\[
\iint_{I \times I} |x - y|^{-t} \, \mathrm{d} \mu_\phi (x) \mathrm{d}
\mu_\phi (y) < c |I|^{-t} \mu_\phi (I)^2,
\]
holds for any $I$, then we could have used this to prove that the set
$E (x, r)$ has a large intersection property, see the proof of
Theorem~\ref{the:gibbs}. However, we are unable to prove that such a
constant exists, and our strategy is instead to prove that we have
such an estimate for sufficiently many intervals to get the dimension
result. The lemma below is what we need.

If $\mathscr{Z}$ is the partition with respect to which $T$ is
piecewise expanding, then the elements of the partition $\mathscr{Z}
\vee T^{-1} \mathscr{Z} \vee \cdots \vee T^{-n+1} \mathscr{Z}$
are called cylinders of generation $n$.

\begin{lemma} \label{lem:energyanddistorsion}
  Let $d_0 > 0$ be given and suppose that \eqref{eq:localenergy} holds
  and that $s < s_0$.  Then there is a constant $K = K (d_0)$ such
  that if $I$ is an interval that is a subset of a cylinder of
  generation $n$ and $|T^n (I)| > d_0$, then
  \begin{equation} \label{eq:goodinterval}
  \iint_{I \times I} |x - y|^{-s} \, \mathrm{d} \mu_\phi (x)
  \mathrm{d} \mu_\phi (y) < K |I|^{-s} \mu_\phi (I)^2.
  \end{equation}
\end{lemma}

\begin{proof}
  Let
  \[
  K_0 = \sup_{|I| > d_0} \frac{|I|^s}{\mu_\phi (I)^2} \iint_{I \times
    I} |x - y|^{-s} \, \mathrm{d} \mu_\phi (x) \mathrm{d} \mu_\phi (y)
  < \infty.
  \]
  By the bounded distortion, there exists a constant $K_1$ such that
  \begin{multline*}
    \frac{|I|^s}{\mu_\phi (I)^2} \iint_{I \times I} |x - y|^{-s} \,
    \mathrm{d} \nu_\phi (x) \mathrm{d} \nu_\phi (y)\\ < K_1
    \frac{|T^n(I)|^s}{\nu_\phi (T^n (I))^2} \iint_{T^n (I) \times
      T^n (I)} |x - y|^{-s} \, \mathrm{d} \nu_\phi (x) \mathrm{d}
    \nu_\phi (y),
  \end{multline*}
  whenever $I$ is an interval contained in a cylinder of generation
  $n$. Since $\mu_\phi = h_\phi \nu_\phi$, where $h_\phi$ is bounded
  and bounded away from zero, the combination of these two estimates
  gives us the desired result.
\end{proof}

By Lemma~\ref{lem:energyanddistorsion} we know that some particular
intervals are good, in the sense that we have the estimate
\eqref{eq:goodinterval}. We will now use these intervals to construct
a Cantor set $N = \cap N_n \subset \limsup E_n$ with large
dimension. The following lemma describes the important properties of
this construction.

\begin{lemma} \label{lem:net}
  Suppose that the assumptions of Lemma~\ref{lem:localfrostman} hold
  with $\mu = \mu_\phi$, and that \eqref{eq:localenergy} is satisfied.
  Then, for any $\varepsilon > 0$, there is a sequence of sets $N_n$
  with the following properties.
  \begin{enumerate}
    \addtocounter{enumi}{1}
  \item[(\roman{enumi})] All $N_n$ are compact, each $N_n = \cup
    N_{n,i}$ is a finite and disjoint union of intervals $N_{n,i}$,
    and $N_{n+1} \subset N_n$.  \addtocounter{enumi}{1}
  \item[(\roman{enumi})] There is an increasing sequence $m_n$ such
    that $N_n \subset E_{m_n}$.  \addtocounter{enumi}{1}
  \item[(\roman{enumi})] For any $N_{n,i}$ we have
    \[
    \sum |U_k|^t \geq \frac{1}{4K} |N_{n,i}|^t,
    \]
    for any cover $\{U_k\}$ of $N_{n,i} \cap N_{n+1}$.
    \addtocounter{enumi}{1}
  \item[(\roman{enumi})] For any $N_{n,i}$ and $N_{n+1,j}$ we have
    \[
    \frac{|N_{n,i}|}{|N_{n+1,j}|} > (4K)^{1/\varepsilon}.
    \]
  \end{enumerate}
\end{lemma}

\begin{proof}
  By Hofbauer \cite{Hofbauer}, Lemma 13, we have that if we choose
  $d_0$ sufficiently small, then the Hausdorff dimension of the set of
  points, for which $|T^n I_n (x)| > d_0$ does not hold for infinitely
  many different $n$, is arbitrarily close to $0$. In particular, if
  we choose $d_0$ sufficiently small, then there is a set $A$ of full
  measure such that for any $x \in A$ there are infinitely many $n$
  with $|T^n I_n (x)| > d_0$.

  If $x \in A$ and $I_n (x)$ has the property that $|T^n I_n (x)| >
  r_0$, then we let $J_{x,n} = I_n (x)$. We denote by $\mathscr{J}$
  the set of all $J_{x,n}$, that is
  \[
  \mathscr{J} = \{\, J_{x,n} : x \in A \, \}.
  \]

  We will define the sets $N_n$ inductively as follows. We set $N_0 =
  [0,1]$. Clearly $[0,1]$ satisfies the assumptions of
  Lemma~\ref{lem:energyanddistorsion}. We let $m_0 = n_{[0,1]}$, where
  $n_{[0,1]}$ is by Lemma~\ref{lem:localfrostman}.

  Suppose that $N_n$ has been defined together with a number $m_n$
  such that for any $N_{n,i}$, Lemma~\ref{lem:localfrostman} is
  satisfied with $n_{N_{n,i}} \leq m_n$.

  We wish to define $N_{n+1}$. The set $A \cap N_n$ has full measure
  in $N_n$. Hence, for any $\varepsilon_n > 0$, we can find a finite
  and disjoint collection $\mathscr{J}_n \subset \mathscr{J}$ such
  that for all $J_{x,n} \in \mathscr{J}_n$ we have $|J_{x,n}| <
  \varepsilon_n$ and $J_{x,n} \subset E_{m_n}$. Moreover, we can
  choose the collection $\mathscr{J}_n$ such that for any $N_{n,i}$,
  if $\mathscr{J}_n'$ denotes the elements of $\mathscr{J}_n$ that are
  subsets of $N_{n,i}$, then
  \[
  \nu_n (\cup \mathscr{J}_n') > \frac{1}{2},
  \]
  where the measure $\nu_n$ is defined by \eqref{eq:nudefinition}. As
  in the proof of Lemma~\ref{lem:localfrostman}, we can then conclude
  that for any $N_{n,i}$ we have
  \[
  \sum |U_k|^t \geq \frac{1}{4}
  K^{-1} |N_{n,i}|^t,
  \]
  for any cover $\{U_k\}$ of $N_{n,i}$.

  We put $N_{n+1} = \cup \mathscr{J}_n$ and $\{N_{n+1,i}\} =
  \mathscr{J}_n$. The number $m_{n+1}$ is taken to be an upper bound
  of $\{\, n_I : I \in J_n \,\}$. By taking $\varepsilon_n$
  sufficiently small we can achieve that
  \[
  \frac{|N_{n,i}|}{|N_{n+1,j}|} > (4K)^{1/\varepsilon}
  \]
  holds for any $N_{n,i}$ and $N_{n+1,j}$.

  By induction, we now get the sets $N_n$ with the desired properties.
\end{proof}

\begin{lemma} \label{lem:dimofcantorset}
  With the assumptions and notation of Lemma~\ref{lem:net} we have
  that $\dimH N \geq t - \varepsilon$, where $N = \cap N_n$.
\end{lemma}

\begin{proof}
  Consider any countable cover $\mathscr{U}=\{U_k\}$ of the set
  $N$. Since $N$ is compact, we can assume that $\mathscr{U}$ is a
  finite cover. We will consider the sum
  \[
  Z_{t-\varepsilon}(\mathscr{U}) = \sum_k |U_k|^{t-\varepsilon},
  \]
  trying to prove that it is uniformly bounded away from 0.

  Step 1. There exists $n_0$ such that there is a finite cover
  $\mathscr{U}' = \{U_k'\}$ of $N$ such that each intersection $U_k'
  \cap N$ is a finite union of $N \cap N_{n_0,i_\ell}$ and
  \begin{equation} \label{eqn:ind1}
    Z_{t- \varepsilon}(\mathscr{U}) \geq \frac 12 Z_{t-\varepsilon}
    (\mathscr{U}').
  \end{equation}
  This can be done by taking $n_0$ so large that the intervals
  $N_{n_0,i}$ are much smaller than all the (finitely many) elements
  of the cover $\mathscr{U}$, and then perturb each $U_k$ so that it
  is aligned with the intervals $N_{n_0,i}$.

  Step 2. Consider a new cover $\mathscr{U}''$, obtained in the
  following way. For any $U_k'$, the set $U_k'\cap N$ must be
  contained in some $N_{n,i}$. There are at most two sets $N_{n+1, j}$
  that intersect $U_k'$ but are not contained in $U_k'$. We replace
  $U_k'$ by at most three open sets: $U_k' \cap N_{n,i} \cap N_{n+1,
    j_1}$, $U_k' \cap N_{n,i} \cap N_{n+1, j_2}$, and $U_k' \cap
  N_{n,i} \setminus \overline{(N_{n+1, j_1} \cup N_{n+1, j_2})}$. The
  latter we leave as is, with the former two we repeat the
  procedure. The end result of this procedure: instead of $U_k'$ we
  have a finite family of open sets $U_\ell''$, each of which contains
  a finite union of $N_{n',i}$ for some $n'$ and does not intersect
  other $N_{n',j}$ (we will call this the {\it wholeness
    property}). We will call such $U_\ell''$ a $n'$-th level element.

  Note that in this subfamily there will be at most one element of
  level $n$ and at most two elements of each level $n'$, $n<n'\leq
  n_0$. The lengths of elements of level $n$ or $n+1$ are not greater
  than of the original $|U_k'|$, and by Lemma~\ref{lem:net}, for any
  element $U_\ell''$ of level $n' \geq n+2$ we have
  \[
  |U_\ell''| \leq (4K)^{-(n'-n-1)/\varepsilon} |U_k'|. 
  \]
  Hence, 
  \[
  \sum_\ell |U_\ell''|^{t-\varepsilon} \leq K' |U_k'|^{t-\varepsilon}
  \]
  for
  \[
  K' = 3 + \frac 2 {(4K)^{(t-\varepsilon)/\varepsilon} -1}.
  \]

  Repeating this procedure for all $U_k'$ and combining the
  subfamilies $\{U_\ell''\}$, we get a new cover $\mathscr{U}''$
  consisting only of the elements with the wholeness property and
  satisfying
  \begin{equation} \label{eqn:ind2}
    Z_{t-\varepsilon}(\mathscr{U}'') \leq K'
    Z_{t-\varepsilon}(\mathscr{U}').
  \end{equation}

  Step 3. Rename $\mathscr{U}''$ by $\mathscr{U}^{(n_0)}$. We remind
  that $n_0$ is the smallest $n$ for which $N_n\subset \bigcup U_k''$
  (that is, the maximal level of elements in $U''$).

  We construct the sequence of covers $\mathscr{U}^{(n)}$ in the
  following way: let $\mathscr{U}^{(n+1)}$ be a cover with the
  wholeness property and with maximal level of elements
  $n+1$. Whenever for some $N_{n,i}$ there are elements
  $U_{k_1}^{(n+1)}, \ldots, U_{k_\ell}^{(n+1)}$ that together cover
  all $N_{n+1,j}\subset N_{n,i}$, we replace those elements by
  $N_{n,i}$. The cover constructed in this way has wholeness property
  and does not have elements of level greater than $n$. Moreover, by
  Lemma \ref{lem:net},
  \begin{equation} \label{eqn:s1}
    |N_{n,i}|^t \leq {4K} \sum |U_{k_i}^{(n+1)}|^t.
  \end{equation}

  Let us divide the elements of $\mathscr{U}^{(n+1)}$ into three
  subcategories. An element $U_k^{(n+1)}$ is called
  \begin{itemize}
  \item {\it simple} if it is of the form $N_{n+1,j}$,
  \item {\it imminent} if it is not simple but of level $n+1$ (hence
    $U_k^{(n+1)} \cap N$ is contained in some $N_{n,j}$),
  \item {\it nonimminent} if it is of level not greater than $n$.
  \end{itemize}
  We divide the sum correspondingly:
  \[
  Z_{t-\varepsilon}(\mathscr{U}^{(n+1)}) =
  Z_{t-\varepsilon}^{(\mathrm{s})}(\mathscr{U}^{(n+1)}) +
  Z_{t-\varepsilon}^{(\mathrm{i})}(\mathscr{U}^{(n+1)}) +
  Z_{t-\varepsilon}^{(\mathrm{n})}(\mathscr{U}^{(n+1)}).
  \]

  Observe that by the construction of $\mathscr{U}^{(n)}$, the simple
  and imminent elements of $\mathscr{U}^{(n+1)}$ are replaced by
  simple elements of $\mathscr{U}^{(n)}$, while the nonimminent
  elements of $\mathscr{U}^{(n+1)}$ pass to $\mathscr{U}^{(n)}$
  unchanged (where some of them become imminent, the other stay
  nonimminent). Hence,
  \begin{equation} \label{eqn:inductive1}
    Z_{t-\varepsilon}^{(\mathrm{i})}(\mathscr{U}^{(n)}) +
    Z_{t-\varepsilon}^{(\mathrm{n})}(\mathscr{U}^{(n)}) =
    Z_{t-\varepsilon}^{(\mathrm{n})}(\mathscr{U}^{(n+1)}).
  \end{equation}

  As for $Z_{t-\varepsilon}^{(\mathrm{s})}(\mathscr{U}^{(n)})$,
  \eqref{eqn:s1} implies
  \[
  Z_{t}^{(\mathrm{s})}(\mathscr{U}^{(n)}) \leq {4K} \bigl(
  Z_{t}^{(\mathrm{s})}(\mathscr{U}^{(n+1)}) +
  Z_{t}^{(\mathrm{i})}(\mathscr{U}^{(n+1)}) \bigr).
  \]

  By Lemma~\ref{lem:net}, if $N_{n+1, i}\subset N_{n,j}$ then
  \[
  \frac 1 {4K} |N_{n+1,i}|^{-\varepsilon} \geq |N_{n,j}|^{-\varepsilon}.
  \]
  hence
  \begin{equation} \label{eqn:inductive2}
    Z_{t-\varepsilon}^{(\mathrm{s})}(\mathscr{U}^{(n)}) \leq
    Z_{t-\varepsilon}^{(\mathrm{s})}(\mathscr{U}^{(n+1)}) + {4K}
    Z_{t-\varepsilon}^{(\mathrm{i})}(\mathscr{U}^{(n+1)}).
  \end{equation}

  Step 4. Induction procedure leads us to the cover $\mathscr{U}^{(0)}
  = \{[0,1]\}$. We have
  \[
  Z_{t-\varepsilon}(\mathscr{U}^{(0)}) =
  Z_{t-\varepsilon}^{(\mathrm{s})} (\mathscr{U}^{(0)}) =
  |[0,1]|^{t-\varepsilon} = 1.
  \]
  Combining equations \eqref{eqn:inductive1} and
  \eqref{eqn:inductive2} and repeating the inductive procedure from
  $n_0$ to 0, we observe that over the procedure, the nonimminent
  element of $\mathscr{U}^{(n_0)}$ first stays nonimminent for some
  time, then it becomes imminent, one step later it is combined into a
  simple element, and then it is combined with other elements into
  another simple element at each step. The only moment in this
  procedure when $Z_{t-\varepsilon}$ can increase is when the imminent
  element is combined into a simple element, which happens at most
  once for each element of $\mathscr{U}^{(n_0)}$. Moreover, at this
  time the corresponding term in the sum $Z_{t-\varepsilon}$ can
  increase at most by a factor $4K$. Hence,
  \[
  Z_{t-\varepsilon}(\mathscr{U}^{(n_0)}) \geq \frac 1 {4K}.
  \]
  Combining this with \eqref{eqn:ind1} and \eqref{eqn:ind2} we get
  \[
  Z_{t-\varepsilon}(\mathscr{U}) \geq \frac 1 {8KK'}.
  \]
  Since the cover $\mathscr{U}$ is arbitrary, it follows that $\dimH N
  \geq t - \varepsilon$.
\end{proof}

\begin{proof}[Proof of Theorem~\ref{the:gibbs}]
  We can now prove Theorem~\ref{the:gibbs} in the same way as
  Theorem~\ref{the:shrinkingtarget}, by replacing the use of
  Lemma~\ref{lem:frostman} with that of
  Lemmata~\ref{lem:localfrostman}, \ref{lem:energyanddistorsion},
  \ref{lem:net} and \ref{lem:dimofcantorset}.

  Since the proof is very similar to that of
  Theorem~\ref{the:shrinkingtarget}, we will only sketch the proof.
  We define the sets $V_n (x)$ and the measures $\mu_{n,x}$ as in the
  proof of Theorem~\ref{the:shrinkingtarget}. We will then have that
  $\mu_{n,x}$ converges weakly to $\mu_\phi$ for $\mu_\phi$-almost
  every $x$.

  We consider the energies $I_s (\mu_{n,x})$ and their expectations
  $\int I_s (\mu_{n,x}) \, \mathrm{d} \mu_\phi$ just as in the proof
  of Theorem~\ref{the:shrinkingtarget} and carry out the same
  estimates.  When we use Lemma~\ref{lem:decay} we need to know that
  \[
  \int |x-y|^{-s} \, \mathrm{d} \mu_\phi (y)
  \]
  is uniformly bounded in $x$. This follows from
  Assumption~\ref{as:measureofball} according to
  Remark~\ref{rem:dimension}.

  In this way we are able to conclude that for $\mu_\phi$-almost all
  $x$, there is a sub-sequence along which the energies
  \[
  \iint |x-y|^{-s} \, \mathrm{d} \mu_\phi (x) \mathrm{d} \mu_\phi (y)
  \]
  are uniformly bounded provided
  \[
  s < \sup \{\, t : \exists c, \forall n : n^{-2} \sum_{j=1}^n
  r_j^{-t} < c \,\}.
  \]
  We can now apply Lemmata~\ref{lem:localfrostman},
  \ref{lem:energyanddistorsion}, \ref{lem:net} and
  \ref{lem:dimofcantorset} to get the desired result on the dimension
  of the set $E (x,r)$.
\end{proof}

\begin{proof}[Proof of Theorem~\ref{the:markov}]
  In the case that $T$ is a Markov map, then we can use
  Lemma~\ref{lem:energyanddistorsion} to conclude that for $t < s$
  there is a constant $K$ such that
  \[
  \iint_{I \times I} |x-y|^{-t} \, \mathrm{d}\mu_\phi (x) \mathrm{d}
  \mu_\phi (y) < K |I|^{-t} \mu_\phi (I)^2
  \]
  holds for any interval $I \subset [0,1]$. Together with
  Lemma~\ref{lem:localfrostman} and what was proved in the proof of
  Theorem~\ref{the:gibbs}, we can conclude that for $\mu_\phi$-almost
  all $x$, whenever $I$ is an interval and $n$ is sufficiently large,
  then any cover $\{U_k\}$ of $E_n \cap I$ satisfies
  \[
  \sum |U_k|^t \geq \frac{1}{2K} |I|^t.
  \]
  This implies, according to Falconer \cite{Falconer}, that $E(x,r)
  \in \mathscr{G}^t$.
\end{proof}

\section*{Acknowledgements}

We thank Roland Zweim\"uller for helping us to find references
\cite{Aaronson} and \cite{AaronsonDenker}.

Micha\l{} Rams was supported by MNiSW grant N201 607640 and National
Science Centre grant 2014/13/B/ST1/01033 (Poland).


\begin{thebibliography}{10}

\bibitem{Aaronson} J. Aaronson, {\em An introduction to infinite
  ergodic theory}, Mathematical Surveys and Monographs 50, American
  Mathematical Society, Providence, RI, 1997, ISBN 0-8218-0494-4.

\bibitem{AaronsonDenker} J. Aaronson, M. Denker, {\em Upper bounds
  for ergodic sums of infinite measure preserving transformations}
  Transactions of the American Mathematical Society 319 (1990), no. 1,
  101--138.

\bibitem{Falconer} K. Falconer, {\em Sets with large intersection
  properties}, Journal of the London Mathematical Society 49 (1994),
  no. 2, 267--280.

\bibitem{Fanetal} A.-H. Fan, J. Schmeling and S. Troubetzkoy, {\em A
  multifractal mass transference principle for Gibbs measures with
  applications to dynamical Diophantine approximation}, Proceedings of
  the London Mathematical Society, 107 (2013), 1173--1219.

\bibitem{Hofbauer} F. Hofbauer, {\em Local dimension for piecewise
  monotonic maps on the interval}, Ergodic Theory and Dynamical
  Systems, 15 (1995), no. 6, 1119--1142.

\bibitem{Lietal} B. Li, B.-W. Wang, J. Wu and J. Xu, {\em The
  shrinking target problem in the dynamical system of continued
  fractions}, Proceedings of the London Mathematical Society, 108
  (2014), 159--186.

\bibitem{LiaoSeuret} L. Liao, S. Seuret, {\em Diophantine
  approximation by orbits of expanding Markov maps}, Ergodic Theory
  and Dynamical Systems, 33 (2013), no. 2, 585--608.

\bibitem{Liverani} C. Liverani, {\em Decay of correlations for
  piecewise expanding maps}, Journal of Statistical Physics, 78
  (1995), no. 3--4, 1111--1129.

\bibitem{Liveranietal} C. Liverani, B. Saussol, S. Vaienti, {\em
  Conformal measure and decay of correlation for covering weighted
  systems}, Ergodic Theory and Dynamical Systems, 18 (1998), no. 6,
  1399--1420.

\bibitem{Persson} T. Persson, {\em A Note on Random Coverings of
  Tori}, Bulletin of the London Mathematical Society, 2015, 47 (1),
  7--12.

\bibitem{PerssonReeve} T. Persson, H. Reeve, {\em A Frostman type
  lemma for sets with large intersections, and an application to
  Diophantine approximation}, Proceedings of the Edinburgh
  Mathematical Society, volume 58, issue 2, June 2015, 521--542.

\bibitem{Rychlik} M. Rychlik, {\em Bounded variation and invariant
  measures}, Studia Mathematica, 76 (1983), no. 1, 69--80.
\end{thebibliography}
\end{document}